\documentclass[11pt,reqno]{amsart}
\usepackage{amsmath,amssymb,graphicx,amscd}
\usepackage{mathrsfs}
\usepackage{mathrsfs}

\renewcommand{\S}{\mathscr{S}}  

\newcommand{\CC}{{\mathbb C}}
\newcommand{\RR}{{\mathbb R}}

\newcommand{\NN}{{\mathbb{N}}}

 \DeclareMathOperator{\Prob}{\mathbb{P}}   
 \DeclareMathOperator{\E}{\mathbb{E}}      
 \DeclareMathOperator{\cov}{cov}             

 \newcommand{\I}{1\!\!1}                   

 \newcommand{\dd}{{\mathrm{d}}}            
 \newcommand{\ii}{{\mathrm{i}}}

\newcommand{\law}{\overset{\mbox{\rm \scriptsize law}}{=}}
\newcommand{\convlaw}{\overset{\mbox{\rm \scriptsize law}}{\longrightarrow}}

\newcommand{\mat}[4]
    {\left(\begin{array}{cc}
    #1  & #2  \\
    #3 &  #4
    \end{array}\right)}

\renewcommand{\Re}{{\mathfrak{Re}}}
\renewcommand{\Im}{{\mathfrak{Im}}}

\newtheorem{thm}{Theorem}[section]
\newtheorem{cor}[thm]{Corollary}
\newtheorem{lem}[thm]{Lemma}
\newtheorem{prop}[thm]{Proposition}
\theoremstyle{definition}

\theoremstyle{remark}
\newtheorem*{rem}{Remark}

\numberwithin{equation}{section}

\begin{document}
\title[The characteristic polynomial of a random unitary matrix]{The characteristic polynomial of a random unitary matrix: a probabilistic approach}

\author{P. Bourgade}
 \address{Laboratoire de Probabilit\'es et Mod\'eles Al\'eatoires \\
Universit\'e Pierre et Marie Curie, et C.N.R.S. UMR 7599 \\ 175,
rue du Chevaleret \\ F-75013 Paris, France}
 \email{bourgade@enst.fr}

\author{C.P. Hughes}
\address{Department of Mathematics,
 University of York,
 York,
 YO10 5DD,
 U.K.}
 \email{ch540@york.ac.uk}

\author{A. Nikeghbali}
 \address{Institut f\"ur Mathematik,
 Universit\"at Z\"urich, Winterthurerstrasse 190,
 CH-8057 Z\"urich,
 Switzerland}
 \email{ashkan.nikeghbali@math.unizh.ch}

\author{M. Yor}
 \address{Laboratoire de Probabilit\'es et Mod\'eles Al\'eatoires,
 Universit\'e Pierre et Marie Curie, et C.N.R.S. UMR 7599,
 175, rue du Chevaleret,
 F-75013 Paris,
 France}

\begin{abstract}
In this paper, we propose a probabilistic approach to the study of
the characteristic polynomial of a random unitary matrix. We
recover the Mellin Fourier transform of such a random polynomial,
first obtained by Keating and Snaith in \cite{KeaSna}, using a
simple recursion formula, and from there we are able to obtain the
joint law of its radial and angular parts in the complex plane. In
particular, we show that the real and imaginary parts of the
logarithm of the characteristic polynomial of a random unitary
matrix can be represented in law as the sum of independent random
variables. From such representations, the celebrated limit theorem
obtained by Keating and Snaith in \cite{KeaSna} is now obtained
from the classical central limit theorems of Probability Theory,
as well as some new estimates for the rate of convergence and law
of the iterated logarithm type results.
\end{abstract}

\maketitle

\section{Introduction} \label{section:Intro}

In \cite{KeaSna}, Keating and Snaith argued that the Riemann zeta
function on the critical line could be modelled by the
characteristic polynomial of a random unitary matrix considered on
the unit circle. In their development of the model they showed,
via calculating the Mellin-Fourier transform, that the logarithm
of the characteristic polynomial weakly converges to a normal
distribution, analogous to Selberg's result on the normal
distribution of values of the logarithm of the Riemann zeta
function \cite{Selberg}.

The purpose of this paper is to prove an equality in law between
the characteristic polynomial and products of independent random
variables. Using this we rederive the limit theorem and
Mellin-Fourier transform of Keating and Snaith and prove some new
results about the speed of convergence.

Let $V_N$ denote a generic $N\times N$ random matrix drawn from
the unitary group $U(N)$ with the Haar measure $\mu_{U(N)}$. The
characteristic polynomial of $V_N$ is
\begin{align*}
Z(V_N,\theta)&:=\det(I_N-e^{-\ii\theta}V_N)\\
&=\prod_{j=1}^N \left(1-e^{\ii(\theta_n - \theta)}\right)
\end{align*}
where $e^{\ii\theta_1},\ldots,e^{\ii\theta_N}$ are the eigenvalues of $V_N$. Note that by
the rotation invariance of Haar measure, if $\theta$ is real then
$Z(V_N,\theta)\law Z(V_N,0)$. Therefore here and in the following
we may simply write $Z_N$ for $Z(V_N,\theta)$. Keating and Snaith
\cite{KeaSna} evaluated the Mellin-Fourier transform of $Z_N$.
Integrating against the Weyl density for Haar measure on $U(N)$,
and using certain Selberg integrals, they obtained, for all $t$
and $s$ with $\Re(t\pm s)>-1$,
\begin{equation}\label{mellinfourier}
\E\left(|Z_N|^t e^{\ii s \arg Z_N}\right) = \prod_{k=1}^N \frac{
\Gamma\left(k\right)\Gamma\left(k+t\right) } {
\Gamma\left(k+\frac{t+s}{2}\right)
\Gamma\left(k+\frac{t-s}{2}\right) }.
\end{equation}
In \cite{KeaSna} and in this article,  $\arg Z_N$ is defined as
the imaginary part of
\[
\log Z_N:=\sum_{n=1}^N\log (1-e^{\ii\theta_n})
\]
with $\Im \log(1-e^{\ii\theta_n})\in (-\pi/2,\pi/2]$. An
equivalent definition for $\log Z_N$ is
the value at point $x=1$ of the unique continuous function
$\log\det(I_N-x V_N)$ (on $[0,1]$) which is 0 at $x=0$.

By calculating the asymptotics of the cumulants of
\eqref{mellinfourier}, they were able to show that for any fixed
$s,t$,
\begin{equation*}
\E\left(|Z_N|^{t/\sqrt{(\log N)/2}} e^{\ii s \arg Z_N / \sqrt{(\log
N)/2}}\right) \to \exp\left(\tfrac12 t^2 - \tfrac12 s^2\right)
\end{equation*}
as $N\to\infty$, and from this deduce the central limit theorem
\begin{equation*}
\frac{\log Z_N}{\sqrt{\frac{1}{2}\log N}} \convlaw
\mathcal{N}_1+\ii \mathcal{N}_2,
\end{equation*}
where $\mathcal{N}_1$ and $\mathcal{N}_2$ are two independent
standard Gaussian random variables.

We will see in the two following sections how (1.1) may be simply
interpreted as an identity in law involving a certain product of
independent random variables. In particular, we shall show that
$\Re\log Z_{N}$ and $\Im\log Z_{N}$ can be written in law as sums
of independent random variables. Sums of independent random variables are very well known
and well studied objects in Probability Theory, and we can thus
have a better understanding of the distribution of the
characteristic polynomial with such a representation. We also
apply the classical limit theorems on such sums to obtain
asymptotic properties of $Z_{N}$ when $N\to\infty$. In particular,
we recover the convergence in law of $\log Z_{N}/\sqrt{\tfrac12
\log N}$ to a standard complex Gaussian law as a consequence of
the classical central limit theorem. We  also obtain some new
results about the rate of convergence and prove an iterated
logarithm law for the characteristic polynomial.

More precisely, the paper is organized as follows: in
Section~\ref{section:Haar}, we use a recursive construction for
the Haar measure on $U\left(N\right)$, to obtain our first
equality in law for the distribution of the characteristic
polynomial as a product of independent random variables, from
which we obtain a new proof of \eqref{mellinfourier} which does
not use Selberg's integrals or the Weyl density. Then in
Section~\ref{section:DecompIndRV} we use \eqref{mellinfourier} to
deduce the joint law of $(\Re\log Z_{N} , \Im \log Z_{N} )$,
writing each component as a sums of independent random variables.
Using these two representations for $Z_N$, in
Section~\ref{section:CLT} two new proofs of Keating-Snaith limit
theorem (the convergence in law of $\log Z_{N}/\sqrt{\tfrac12\log
N}$ to a standard complex Gaussian law) are provided. We also give
estimates on the rate of convergence in the central limit theorem.
In Section~\ref{section:ILL}  we see how strong limit theorems
such as the iterated logarithm can be deduced from our
representations.

In a companion paper to \cite{KeaSna}, Keating and Snaith
\cite{KeaSna2} studied the characteristic polynomial for classical
compact groups other than the unitary group, and in
Section~\ref{section:OrthogGroup}, we also give similar results
for $SO(2N)$ which plays a similar role to $U(N)$ for other
families of $L$-functions.

Since the publication of \cite{KeaSna}, there have been many
developments in understanding the distribution of the
characteristic polynomial. For example, other limit theorems and
large deviation results were derived for the characteristic
polynomial by Hughes, Keating and O'Connell in \cite{HKO}. The
distribution of the characteristic polynomial away from the point
$\theta=0$ in the other groups has been studied by Odgers
\cite{Odgers}. For more details about the connections between
random matrix theory and analytic number theory, see \cite{MezSna}
and the references therein, or the excellent survey article by
Royer \cite{Royer}.

\section{Decomposition of the Haar measure}\label{section:Haar}

In this section, we give an alternative proof of formula
\eqref{mellinfourier}, which does not necessitate the explicit
knowledge of the Weyl density formula and the values of some
Selberg integrals. This new demonstration relies on a recursive
presentation of the Haar measure $\mu_{U(N)}$.

\subsection{Decomposition of the Haar measure} \label{subsection:DecompHaar}

Let $V_N$ be distributed with Haar measure $\mu_{U(N)}$ on $U(N)$.
If $M\in U(N+1)$ is independent of $V_N$ a natural question to ask
is under which condition on the distribution of $M$, is the matrix
\begin{equation}\label{eq:matrix_decomp}
M\left(
\begin{array}{cc}
1&0\\
0&V_N
\end{array}
\right)
\end{equation}
distributed with the Haar measure in dimension $N+1$ ?

The solution to such a question allows one to recursively build up
a Haar distributed element of $U(N)$.

This question was partially answered by
Mezzadri relying on a general method due to Diaconis and
Shahshahani \cite{DiacShah}. Due to a factorization of the Ginibre
ensemble, Mezzadri showed that when $M$ is a suitable Householder
reflection then \eqref{eq:matrix_decomp} is distributed with the
Haar measure in dimension $N+1$. More precisely, he showed that if
$v$ is a unit vector chosen uniformly on the $(N+1)$-dimensional
unit complex-sphere
\[
\S_\CC^{N+1}:=\{(c_1,\dots,c_{N+1})\in\CC^{N+1} \ : \
|c_1|^2+\dots+|c_{N+1}|^2=1\} ,
\]
and if $\theta$ is the argument of the first coordinate of $v$,
and $u$ is the unit vector along the bisector of $e_1$ and $v$,
where $e_1=(1,0,\dots,0)$ is the unit vector for the first
coordinate, then one could take $M$ to be an element of $U(N+1)$
that can be written $-e^{-\ii\theta}(I_{N+1}-2u\overline{u}^T)$.

By the application $M\mapsto M_1$, it is clear that a necessary
condition for our question must be that $v$ must be distributed
according to the uniform measure on $\S_\CC^{N+1}$. The following
proposition states that this condition is also sufficient. It is a
slight generalization of Mezzadri's result, and its proof does not
require a decomposition of the Ginibre ensemble.

\begin{prop} \label{prop:HaarDecomp}
Let $M\in U(N+1)$ be such that its first column $M_1$ is uniformly
distributed on $\S_\CC^{N+1}$. If $V_N\in U(N)$ is chosen
independently of $M$ according to the Haar measure $\mu_{U(N)}$,
then the matrix
\[
V_{N+1} := M \mat{1}{0}{0}{V_N}
\]
is distributed with the Haar measure $\mu_{U(N+1)}$.

\end{prop}

\begin{proof}
Due to the uniqueness property of the Haar measure, we only need
to show that for a fixed $U\in U(N+1)$
\[
U M \mat{1}{0}{0}{V_N} \law M \mat{1}{0}{0}{V_N}.
\]

In the following, a matrix $A$ will often be written $(A_1 \|
\tilde{A})$, where $A_1$ is its first column. As $U\in U(N+1)$,
$(UM)_1=UM_1$ is distributed uniformly on the complex unit sphere
$\S_\CC^{N+1}$, so we can write $UM=(P_1 \| \tilde{P})$, with
$P_1$ uniformly distributed on $\S_\CC^{N+1}$ and $\tilde{P}$
having a distribution on the orthogonal hyperplane of $P_1$. We
then need to show that
\[
(P_1\|\tilde{P}) \mat{1}{0}{0}{V_N} \law (M_1\|\tilde{M})
\mat{1}{0}{0}{V_N},
\]
where all matrices are still independent. As $M_1$ and $P_1$ are
identically distributed, by conditioning on $M_1=P_1=v$ (here $v$ is
any fixed element of $\S_\CC^{N+1}$) it is
sufficient to show that
\[
(v\|P') \mat{1}{0}{0}{V_N} \law (v\|M') \mat{1}{0}{0}{V_N},
\]
where $M'$ (resp $P'$) is distributed like $\tilde{M}$ (resp
$\tilde{P}$) conditionally to $M_1=v$ (resp $P_1=v$). Let $A$ be any
element of $U(N+1)$ such that $A(v)=(1,0,\dots,0)$. Since $A$ is
invertible, we just need to show that
\[
A(v \| P') \mat{1}{0}{0}{V_N} \law A(v \| M') \mat{1}{0}{0}{V_N},
\]
that is to say
\[
P''V_N \law M''V_N,
\]
where $P''$ and $M''$ are distributed on $U(N)$ independently of
$V_N$. By independence and conditioning on $P''$ (resp $M''$), we
get $P''V_N \law V_N$ (resp $M''V_N \law V_N$) by definition of
the Haar measure $\mu_{U\left(N\right)}$. This gives the desired
result.
\end{proof}

The result of this proposition is very natural. It states, roughly
speaking, that in order to choose uniformly an element of $U(N+1)$
(that is to say an orthogonal unitary basis) one just needs to
choose the first element uniformly on the sphere and then an
element of $U(N)$ in the hyperplane orthogonal to the first
element, uniformly.

\subsection{Recovering the Mellin Fourier transform} \label{subsection:RecoverMellinFourier}

The decomposition of the Haar measure presented in the previous
paragraph gives another proof for equation \eqref{mellinfourier}.
In reality, the following Proposition~\ref{prop:Decomp1} gives
much more, as we get a representation of $Z(V_N)$ as a product of
$N$ simple independent random variables.

\begin{prop} \label{prop:Decomp1}
Let $V_N\in U(N)$ be distributed with the
Haar measure $\mu_{U(N)}$. Then for all $\theta\in\RR$
$$
\det (I_N-e^{i\theta}V_N)\law \prod_{k=1}^N
\left(1+e^{\ii\theta_k}\sqrt{\beta_{1,k-1}}\right),
$$
with $\theta_1,\dots,\theta_n,\beta_{1,0},\dots,\beta_{1,n-1}$
independent random variables, the $\theta_k$'s uniformly
distributed on $[0,2\pi]$ and the $\beta_{1,j}$'s ($0\leq j\leq
N-1$) being beta distributed with parameters 1 and $j$ (by
convention, $\beta_{1,0}$ is the Dirac distribution on 1).
\end{prop}

\begin{proof}
As previously mentioned, it suffices to consider the
case $\theta=0$.

Note now that in Proposition~\ref{prop:HaarDecomp}, we can choose
any matrix $M\in U(N)$ with $M_1$ uniformly distributed on the
complex sphere $\S_\CC^{N}$. Let us choose the simplest suitable
transformation $M$ : the reflection with respect to the median
hyperplane of $e_1$ and $M_1$, where $M_1$ is chosen uniformly on
$\S_\CC^{N}$. Let the vector $v$ be $M_1-e_1$. Therefore there
exists $(\lambda_2,\dots,\lambda_{N})\in\CC^{N-1}$ such that
$$M=\left(e_1+v\| e_2+\lambda_2 v\|\dots\| e_{N}+\lambda_{N} v\right).$$
So, with Proposition~\ref{prop:HaarDecomp}, one can write
\begin{align*}
\det(I_N-V_N) & \law
\det\left[I_N-M\left(\begin{array}{cc}1&0\\0&V_{N-1}\end{array}\right)\right]
\\
&=
\det\left[\left(\begin{array}{cc}1&0\\0&\overline{V}^T_{N-1}\end{array}\right)-M\right]
\det\left(\begin{array}{cc}1&0\\0&V_{N-1}\end{array}\right).
\end{align*}
If we call $(u_1\|\dots\|u_{N-1}):=\overline{V}^T_{N-1}$ then using
the multi-linearity of the determinant we get
\begin{align*}
\det&\left[\left(\begin{array}{cc}1&0\\0&\overline{V}^T_{N-1}\end{array}\right)-M\right]
\\
&= \det\left(
-v\|\left(\begin{array}{c}0\\u_1\end{array}\right)-e_2-\lambda_2
v\|\dots\|\left(\begin{array}{c}0\\u_{N-1}\end{array}\right)-e_{N}-\lambda_{N}
v
\right)\\
&= \det\left(
-v\|\left(\begin{array}{c}0\\u_1\end{array}\right)-e_2\|\dots\|\left(\begin{array}{c}0\\u_{N-1}\end{array}\right)-e_{N}
\right)\\
&= \det\left(
\begin{array}{c|c}
-v_1&0\\
\dots&\overline{V}^T_{N-1}-I_{N-1}
\end{array}
\right)\\
&=-v_1\det\left(\overline{V}^T_{N-1}-I_{N-1}\right).
\end{align*}
Finally,
$$\det(I_N-V_N)\law -v_1 \det(I_{N-1}-V_{N-1}),$$
with $-v_1=1-M_{11}$ and $V_{N-1}$ independent. Therefore, to
prove Proposition~\ref{prop:Decomp1}, we only need to show that
$M_{11}\law e^{\ii\theta_N}\sqrt{\beta_{1,N-1}}$. This is
straightforward because, since $M_1$ is a random vector chosen
uniformly on $\S_\CC^N$,  we know that
$$M_{11}\law\frac{x_1+\ii y_1}{\sqrt{(x_1^2+y_1^2)+\dots+(x_N^2+y_N^2)}}\law e^{\ii\theta_N}\sqrt{\beta_{1,N-1}},$$
with the $x_i$'s and $y_i$'s all independent standard normal
variables, $\theta_N$ and $\beta_{1,N-1}$ as stated in
Proposition~\ref{prop:Decomp1}.
\end{proof}

To end the proof of \eqref{mellinfourier}, we now only need the
following lemma.

\begin{lem} \label{lem:MellinFourier}
Let $X:=1+e^{\ii\theta}\sqrt{\beta}$, where $\theta$ has uniform
distribution on $[0,2\pi]$ and, independently $\beta$ has a beta
law with parameters 1 and $N-1$. Then, for all $t$ and $s$ with
$\Re(t\pm s)>-1$
\begin{equation*}
\E\left(|X|^t e^{\ii s \arg X}\right) = \frac{ \Gamma\left(N\right)
\Gamma\left(N+t\right) }{ \Gamma\left(N+\frac{t+s}{2}\right)
\Gamma\left(N+\frac{t-s}{2}\right) }.
\end{equation*}
\end{lem}

\begin{proof}
First, note that
\begin{align*}
\E\left(|X|^t e^{\ii s \arg X}\right) &=
\E\left(X^{(t+s)/2}\overline{X}^{(t-s)/2}\right)\\
&=
\E\left(\left(1+e^{\ii\theta}\sqrt{\beta}\right)^a\left(1+e^{-\ii\theta}\sqrt{\beta}\right)^b\right),
\end{align*}
with $a=(t+s)/2$ and $b=(t-s)/2$. Recall that if $|x|<1$ and $u\in
\RR$ then
$$(1+x)^u=\sum_{k=0}^\infty\frac{u(u-1)\dots(u-k+1)}{k!} x^k=\sum_{k=0}^\infty\frac{(-1)^k(-u)_k}{k!} x^k,$$
where $(y)_k=y(y+1)\dots(y+k-1)$ is the Pochhammer symbol. As
$|e^{\ii\theta}\sqrt{\beta}|<1$ a.s., we get
\begin{multline*}
\E\left[|X|^t e^{\ii s \arg X}\right] \\
= \E\left[ \left(
\sum_{k=0}^\infty\frac{(-1)^k(-a)_k}{k!}
\beta^{\frac{k}{2}}e^{\ii k\theta} \right) \left(
\sum_{\ell =0}^\infty\frac{(-1)^\ell (-b)_\ell }{\ell !}
\beta^{\frac{\ell }{2}}e^{-\ii\ell \theta} \right) \right].
\end{multline*}
After an expansion of this double sum (it is absolutely convergent
because $\Re(t\pm  s)>-1$), all terms with $k\neq \ell $ will give
an expectation equal to 0, because we integrate with respect to
the uniform probability measure along the unit circle. So we get
$$
\E\left[|X|^t e^{\ii s \arg X}\right]= \E \left[
\sum_{k=0}^\infty\frac{(-a)_k (-b)_k}{(k!)^2} \beta^k \right].
$$
As $\beta$ is a beta variable with parameters 1 and $N-1$, we have
$$\E\left[\beta^k\right]=\frac{\Gamma(1+k)\Gamma(N)}{\Gamma(1)\Gamma(N+k)}=\frac{k!}{(N)_k},$$
hence
$$
\E\left[|X|^t e^{\ii s \arg X}\right]= \sum_{k=0}^\infty\frac{(-a)_k
(-b)_k}{k! (N)_k}.
$$
Note that this series is equal to the value at $z=1$ of the
hypergeometric function $H(-a,-b,N;z)$. This value is well known
(see, for example, \cite{AAR}) and yields:
$$
\E\left[|X|^t e^{\ii s \arg
X}\right]=\frac{\Gamma(N)\Gamma(N+a+b)}{\Gamma(N+a)\Gamma(N+b)} .
$$
This gives the desired result.
\end{proof}

{\it Comments about Selberg integrals.} To prove
\eqref{mellinfourier} Keating and Snaith \cite{KeaSna}, relying on
Weyl's integration formula, used the result by Selberg
\begin{multline}\label{eq:SelbergInt}
J(a,b,\alpha,\beta,\gamma,N):=\\
\idotsint_{\RR^N} \prod_{1\leq j<\ell\leq N}\left|
x_j-x_\ell\right|^{2\gamma}
\prod_{j=1}^N(a+ i x_j)^{-\alpha}(b- i x_j)^{-\beta}\dd x_j \\
= \frac{(2\pi)^N}{(a+b)^{(\alpha+\beta)N-\gamma N(N-1)-N}}
\prod_{j=0}^{N-1} \frac{ \Gamma(1+\gamma+j\gamma)
\Gamma(\alpha+\beta_(N+j-1)\gamma-1) }{ \Gamma(1+\gamma)
\Gamma(\alpha-j\gamma) \Gamma(\beta-j\gamma) },
\end{multline}
in the specific case $a=b=1$ and $\gamma=1$. Thus our
probabilistic proof for \eqref{mellinfourier} also gives an
alternative proof of \eqref{eq:SelbergInt} for these specific
values of the parameters. Moreover, as we will see in
Section~\ref{section:OrthogGroup}, a similar result as
Proposition~\ref{prop:Decomp1} still holds for the orthogonal
group. As a consequence Selberg's integral formula admits a
probabilistic proof for $a=b$ and $\gamma=1/2$ and $1$ (however
this method relies on Weyl's integration formula, which is
essentially analytical).

\subsection{Decomposition of the characteristic polynomial off the unit circle.}
\label{subsection:DecompCharPoly}

Proposition~\ref{prop:Decomp1} can be extended to the law of the
characteristic polynomial of a random unitary matrix off the unit
circle where we replace $e^{i\theta}$ by a fixed $x$. Once more,
due to the rotation invariance of the unitary group, we may take
$x$ to be real.

\begin{prop} \label{prop:DecompOffCircle}
Let $x\in\RR$, $V_{N-1}$ distributed with the Haar measure $\mu_{U(N-1)}$,
$M_1$ uniformly chosen on $\S_\CC^{N}$, independently of $U_{N-1}$. We write $M_{11}$ for
the first coordinate of $M_1$, and $\tilde{M_1}$ for the vector with coordinates $M_{12},\dots,M_{1N}$.

Then, if $V_N$ is distributed with the Haar measure $\mu_{U(N)}$,
\begin{multline}\label{eq:DecompCharPolyOff}
\det(I_N-x V_N) \law (1-x M_{11})\ \det(I_{N-1}-xV_{N-1}) \\
+\frac{x(1-x)}{1-\overline{M_{11}}} \overline{\tilde{M}_1}^T
(\overline{V_{N-1}}^T-x I_{N-1})^{-1}\tilde{M}_1\ \det(I_{N-1}-x
V_{N-1}).
\end{multline}
\end{prop}

\begin{proof}
The idea is the same as for Proposition~\ref{prop:Decomp1}, where
we use Proposition~\ref{prop:HaarDecomp} with a specific choice of
$M$, the reflection with respect to the hyperplane median to
$e_1:=(1,0,\dots,0)$ and $M_1$, a vector of $\S_\CC^{N}$ chosen
uniformly. If $k:=M_1-e_1$ we can write more precisely
$$
M=\left(M_1,e_2+\frac{-\overline{M}_{12}}{1-\overline{M}_{11}} k,\dots,e_N+\frac{-\overline{M}_{1N}}{1-\overline{M}_{11}} k\right).
$$
Thus, using multi-linearity of the determinant, due to
Proposition~\ref{prop:HaarDecomp} we get after some
straightforward calculation
\begin{align*}
\det(I_N-x V_N)&\law
\det\left[\left(\begin{array}{cc}1&0\\0&\overline{V}^T_{N-1}\end{array}\right)-x M\right]
\det\left(\begin{array}{cc}1&0\\0&V_{N-1}\end{array}\right)\\
&=b\
\det\left(\begin{array}{cc}a&\overline{\tilde{M}}^T_1\\\tilde{M}_1&\overline{V}^T_{N-1}-x
I_{N-1}\end{array}\right)
\det\left(\begin{array}{cc}1&0\\0&V_{N-1}\end{array}\right)\\
\end{align*}
with $b=\frac{-x(1-x)}{1-\overline{M}_{11}}$ and $a=\frac{(1-x M_{11})(1-\overline{M}_{11})}{-x(1-x)}$.
As we want to express these terms with respect to $\det(I_{N-1}-V_{N-1})$, writing $B:=\overline{V}^T_{N-1}-x I_{N-1}$ leads to
\begin{align*}
\det(I_N-x V_N)&\law b\
\det\left(\begin{array}{cc}a&\overline{\tilde{M}}^T_1\\\tilde{M}_1&
B\end{array}\right)
\det\left(\begin{array}{cc}1&0\\-B^{-1}\tilde{M}_1&V_{N-1}\end{array}\right)\\
&=b\ \det\left(\begin{array}{cc}a-\overline{\tilde{M}}^T_1B^{-1}\tilde{M}_1&\cdots\\0& B V_{N-1}\end{array}\right)\\
&=b\ (a-\overline{\tilde{M}}^T_1B^{-1}\tilde{M}_1)\ \det(I_{N-1}-xV_{N-1}).
\end{align*}
This is the expected result.
\end{proof}

One may try to get a new proof of Weyl's integration formula
thanks to this recursive construction of a characteristic
polynomial. Let $\nu_N$ be the probability measure on $[0,2\pi)^N$
with density
\[
\nu_N(\dd a_1,\dots,\dd a_N)=c_N\prod_{j\neq k}\left|e^{i
a_j}-e^{i a_k}\right|^2\dd a_1\dots\dd a_N.
\]
It would be sufficient to show that if $(\theta_1,\dots,\theta_N)$
and $(\tilde{\theta}_1,\dots,\tilde{\theta}_{N-1})$ are
independent and respectively distributed according to $\nu_N$ and
$\nu_{N-1}$, then for all $x\in\RR$, with the notations of the
proposition,
\begin{multline*}
\prod_{j=1}^{N}(1-xe^{i\theta_j}) \law(1-x M_{11})
\prod_{j=1}^{N-1}(1-xe^{i\tilde{\theta}_j}) \\
+\frac{x(1-x)}{1-\overline{M_{11}}}\sum_{j=1}^{N-1}e^{i\tilde{\theta}_j}|M_{1,j+1}|^2\prod_{k\neq
j}\left(1-x e^{i\tilde{\theta}_k}\right).
\end{multline*}
However, this identity in law does not seem to have an easy direct
explanation.

\section{Decomposition into independent random variables}
\label{section:DecompIndRV}

\subsection{Some formulae about the beta-gamma algebra}
\label{subsection:BetaGammaAlgebra}

We recall here some well known facts about the beta-gamma algebra
which we shall often use in the sequel. A gamma random variable
$\gamma_{a}$ with coefficient $a>0$ has density given by:
\begin{equation*}
\Prob\left\{\gamma_{a}\in \dd
t\right\}=\frac{t^{a-1}}{\Gamma\left(a\right)}e^{-t}\dd t.
\end{equation*}
Its Mellin transform is ($s>0$)
\begin{equation*}
\E\left[\gamma_{a}^{s}\right]=\frac{\Gamma\left(a+s\right)}{\Gamma\left(a\right)}.
\end{equation*}
A beta random variable $\beta_{a,b}$ with strictly positive
coefficients $a$ and $b$ has density on $\left[0,1\right]$ given
by
\begin{equation*}
\Prob\left\{\beta_{a,b}\in \dd
t\right\}=\frac{\Gamma\left(a+b\right)}{\Gamma\left(a\right)\Gamma\left(b\right)}t^{a-1}\left(1-t\right)^{b-1}\dd
t.
\end{equation*}
Its Mellin transform is ($s>0$)
\begin{equation}\label{eq:mmts_beta}
\E\left[\beta_{a,b}^{s}\right] = \frac{\Gamma\left(a+s\right)}
{\Gamma\left(a\right)} \frac{\Gamma\left(a+b\right)}
{\Gamma\left(a+b+s\right)}.
\end{equation}

We will also make use of the following two properties (see
\cite{ChauYor} for justifications): the algebra property where all
variables are independent
\begin{equation*}
\beta_{a,b}\gamma_{a+b}\law\gamma_{a},
\end{equation*}
and the duplication formula for the gamma variables, with all
variables independent
$$\gamma_{j}\law 2\sqrt{\gamma_{\frac{j}{2}}\gamma'_{\frac{j+1}{2}}}.$$

\subsection{The joint law of $(|Z_{N}| , \Im\log Z_{N} )$}
\label{section:JointLawZ_N}

In this Section, we use the Mellin Fourier transform
\eqref{mellinfourier} obtained in Section~\ref{section:Haar} to
deduce the joint law of $(|Z_{N}| , \Im\log Z_{N} )$. For
simplicity, let us write
\[
\Delta_{N}\equiv|Z_{N}|, \mathrm{and}\;\; I_{N}\equiv\Im\log Z_{N}
\]
so with this notation formula \eqref{mellinfourier} states
\begin{equation}\label{Mellinfourier2}
    \E\left[\Delta_{N}^t e^{i s I_{N}}\right]=\prod_{k=1}^N\frac{\Gamma\left(k\right)\Gamma\left(k+t\right)}{\Gamma\left(k+\frac{t+s}{2}\right)\Gamma\left(k+\frac{t-s}{2}\right)}.
\end{equation}

\begin{lem}
Let $W_j$ have density
\[
K_j \cos(v)^{2(j-1)} \I_{(-\pi/2,\pi/2)},
\]
where
\[
K_j = \frac{2^{2\left(j-1\right)}
\left(\left(j-1\right)!\right)^{2}} {\pi \left(2j-2\right)!} ,
\]
and let
\[
X := \beta_{j,j-1} 2 \cos W_j e^{\ii W_j} .
\] where all the random variables in sight are independent.
Then
\begin{equation}
\E\left[|X|^t  e^{\ii s \arg X} \right] = \frac{\Gamma(j)
\Gamma(j+t)}{ \Gamma(j+(t+s)/2) \Gamma(j+(t-s)/2)}
\end{equation}
\end{lem}

\begin{proof}
By the definition of $X$, we have that
\begin{multline*}
\E\left[ |X|^t e^{\ii s \arg X} \right] = \\
\E\left[ (\beta_{j,j-1})^t \right] K_j \int_{-\pi/2}^{\pi/2}
e^{\ii s x} \left(e^{\ii x} + e^{-\ii x}\right)^t  \left(e^{\ii x}
+ e^{-\ii x}\right)^{2(j-1)} \;\dd x
\end{multline*}
By \eqref{eq:mmts_beta} we have
\[
\E\left[ (\beta_{j,j-1})^t \right] = \frac{\Gamma(j+t)
\Gamma(2j-1)}{\Gamma(j)\Gamma(2j-1+t)}
\]
Note that
\begin{multline*}
e^{\ii s x} \left(e^{\ii x} + e^{-\ii x}\right)^{2(j-1)}
\left(e^{\ii x} + e^{-\ii x}\right)^t \\
= \left(1+e^{2\ii
x}\right)^{j-1+(t+s)/2} \left(1+e^{-2\ii x}\right)^{j-1+(t-s)/2}
\end{multline*}
Both terms on the RHS can be expanded as a series in $e^{2 \ii x}$
or $e^{-2\ii x}$ for all $x$ other than $x=0$. Integrating over
$x$ between $-\pi/2$ and $\pi/2$, only the diagonal terms survive,
and so
\begin{multline*}
\int_{-\pi/2}^{\pi/2} e^{\ii s x} \left(e^{\ii x} + e^{-\ii
x}\right)^{2(j-1)} \left(e^{\ii x} + e^{-\ii x}\right)^t \;\dd x
\\
= \sum_{k=0}^\infty \frac{\left(-(j-1+(t+s)/2)\right)_k
\left(-(j-1+(t-s)/2)\right)_k}{k! k!} \\
=H(-(j-1+(t+s)/2 ),-( j-1+(t-s)/2), 1; 1)
\end{multline*}
where $H$ is a hypergeometric function. The values of
hypergeometric functions at $z=1$ are well known (see, for
example \cite{AAR}), and are given by
\[
\frac{\Gamma(2j-1+t) }{\Gamma(j+(t+s)/2) \Gamma(j+(t-s)/2)}
\]
and this completes the proof.
\end{proof}

The next Theorem now follows easily from the previous lemma:
\begin{thm}\label{thmfondamentaldedecomposition}
Let $\Delta_{N}\equiv|Z_{N}|$ and $ I_{N}\equiv\Im\log Z_{N} $.
Let $\left(\beta_{j,j-1}\right)_{1\leq j\leq N}$ be independent
beta variables of parameters $j$ and $j-1$ respectively (with the
convention that $\beta_{1,0}\equiv1$). Define
$W_{1},\ldots,W_{N}$ as independent random variables which are
independent of the $\left(\beta_{j,j-1}\right)_{1\leq j\leq N}$,
with $W_{j}$ having the density:
\begin{equation}\label{Eq1}
 \sigma_{2\left(j-1\right)}\left(\dd
v\right) =
\dfrac{2^{2\left(j-1\right)}\left(\left(j-1\right)!\right)^{2}}{\pi \left(2j-2\right)!}\cos^{2\left(j-1\right)}\left(v\right)\I_{\left(\frac{-\pi}{2},\frac{\pi}{2}\right)}\dd
v.
\end{equation}
Then, the joint distribution of $\left(I_{N},\Delta_{N}\right)$ is
given by:
\begin{equation}\label{loijointe}
    \left(I_{N},\Delta_{N}\right)\law \left(\sum_{j=1}^{N}W_{j},\prod_{j=1}^{N}\beta_{j,j-1}2\cos W_{j}\right).
\end{equation}
\end{thm}

We now recover a formula obtained in \cite{NikYor} in the study of
the relations between the Barnes function and generalized gamma
variables. To this end, we need the following elementary lemma
whose proof is left to the reader:
\begin{lem}
Let $V_{t}$ be distributed as
\[
\mathbb{P}\left(V_{t}\in\dd
v\right)=\dfrac{\left(2\cos\left(v\right)\right)^{t}}{\pi
K_{t}},\quad |v|<\frac{\pi}{2},
\]
with $K_{t}=\frac{\Gamma\left(1+t\right)}
{\left(\Gamma\left(1+\frac{t}{2}\right)\right)^{2}}$. Then
\[
\cos\left(V_{t}\right) \law
\sqrt{\beta_{\frac{t+1}{2},\frac{1}{2}}} .
\]
If $W_{j} \law V_{2\left(j-1\right)}$ then
\begin{equation}\label{eq00}
\cos\left(W_{j}\right) \law
\sqrt{\beta_{j-\frac{1}{2},\frac{1}{2}}}.
\end{equation}
\end{lem}
\begin{prop}[\cite{NikYor}] \label{propNY}
Let $\left(\gamma_{j}\right)_{1\leq j\leq N}$ and
$\left(\gamma_{j}'\right)_{1\leq j\leq N}$ be sequences of
independent gamma(j) variables. Then we have
\begin{equation}\label{eqbarnes}
\prod_{j=1}^{N}\gamma_{j} \law \Delta_{N} \prod_{j=1}^{N}
\sqrt{\gamma_{j} \gamma'_{j} }.
\end{equation}
\end{prop}
\begin{proof}
Considering only the second component in \eqref{loijointe}, and
multiplying both sides by $\prod_{j=1}^{N}\gamma_{2j-1}$, and
thanks to the beta-gamma algebra, we obtain:
\begin{equation}\label{equa}
\Delta_{N}\left(\prod_{j=1}^{N}\gamma_{2j-1}\right) \law
\left(\prod_{j=1}^{N}\gamma_{j}\right)2^{N}\left(\prod_{j=1}^{N}\cos\left(W_{j}\right)\right).
\end{equation}
Now we apply the lemma to the right hand side of the above equality
in law to obtain that
\begin{equation}\label{equat}
\left( \prod_{j=1}^{N} \gamma_{j} \right) 2^{N} \left(
\prod_{j=1}^{N} \cos\left( W_{j} \right) \right) \law \left(
\prod_{j=1}^{N} \gamma_{j} \right) 2^{N} \prod_{j=1}^{N}
\sqrt{\beta_{j-\frac{1}{2},\frac{1}{2}}}.
\end{equation}
On the other hand, from the duplication formula for the gamma
function, we have for any $a>0$
\[
\gamma_{a} \law 2 \sqrt{ \gamma_{\frac{a}{2}}
\gamma'_{\frac{a+1}{2}} },
\]
thus, on the left hand side of \eqref{equa} we get
\begin{equation}\label{equat1}
\Delta_{N} 2^{N} \prod_{j=1}^{N} \sqrt{ \gamma_{\frac{2j-1}{2}}
\gamma'_{j}} \law \Delta_{N} 2^{N} \prod_{j=1}^{N} \sqrt{
\beta_{j-\frac{1}{2},\frac{1}{2}} \gamma_{j} \gamma'_{j} }.
\end{equation}
Now comparing \eqref{equat} and \eqref{equat1} we obtain
\[
\prod_{j=1}^{N} \gamma_{j} \law \Delta_{N} \prod_{j=1}^{N} \sqrt{
\gamma_{j} \gamma'_{j} } .
\]
\end{proof}
Infinitely divisible laws form a very remarkable and well studied
family of laws in Probability Theory. It is easily see from
Proposition~\ref{propNY} that the law of $\log|Z_{N}|$  is
infinitely divisible.
\begin{prop}
The law of $\log|Z_{N}|$  is infinitely divisible.
\end{prop}
\begin{proof}
It follows from Proposition~\ref{propNY} and the fact that
logarithm of a gamma variable is infinitely divisible (see, for
example, \cite{CPY}).
\end{proof}
\begin{rem}
The law of $\Im\log Z_{N}$ is not infinitely divisible since it is a bounded random variable.
\end{rem}

\section{Central limit theorems} \label{section:CLT}

In this section, we give two alternative proofs of the following
central limit theorem by Keating and Snaith \cite{KeaSna}. The
first one from the decomposition in Section~\ref{section:Haar},
the second from the last decomposition in
Section~\ref{section:DecompIndRV}. The original proof by Keating
and Snaith relies on an expansion of formula \eqref{mellinfourier}
with cumulants.

\begin{thm}\label{thm:CharPolyCLT}
Let $Z_N:=\det(I_N-V_N)$, where $V_N$ is distributed with the Haar measure on the unitary group $U(N)$.
Then,
\begin{equation}
\frac{\log Z_N}{\sqrt{\frac{1}{2}\log N}} \convlaw \mathcal{N}_1 +
\ii \mathcal{N}_2,
\end{equation}
as $N\to\infty$, with $\mathcal{N}_1$ and $\mathcal{N}_2$ independent standard normal variables.
\end{thm}
\rm

\subsection{Proof from the decomposition in Section~\ref{section:Haar}.}
\label{subsection:CLTFirstProof}

From Proposition~\ref{prop:Decomp1}, we know that
$$
\det (I_N-V_N)\law \prod_{k=1}^N
\left(1+e^{\ii\theta_k}\sqrt{\beta_{1,k-1}}\right),
$$
with $\theta_1,\dots,\theta_N,\beta_{1,0},\dots,\beta_{1,N-1}$
independent random variables, the $\theta_k$'s uniformly
distributed on $[0,2\pi]$ and the $\beta_{1,j}$'s ($0\leq j\leq
N-1$) being beta distributed with parameters 1 and $j$.

In the following we note
$$X_N:=\sum_{i=1}^N\log\left(1+e^{\ii\theta_k}\sqrt{\beta_{1,k-1}}\right),$$
with $\log(1+\epsilon)$ defined here as $\sum_{j\geq 1}
(-1)^{j+1}\epsilon^j/j$ (this is convergent a.s. because
$|e^{\ii\theta_k}\sqrt{\beta_{1,k-1}}|<1$ a.s.).

What we need in order to prove Theorem~\ref{thm:CharPolyCLT} is
the following :
\begin{enumerate}
\item first to show that $X_N$ is equal in law to $\log Z_N$, as it is defined in the introduction.
It is not so obvious, because
the imaginary parts could have a  $2k\pi$ difference.
\item then show that $X_N$ converges to what is expected in (4.1).
\end{enumerate}

{\it Proof for (1).} Equation~\eqref{eq:DecompCharPolyOff}, stated
for a fixed $x$, is also obviously also true for a trajectory, for
example for $x\in[0,1]$,
\begin{multline*}
 (\det(I_N-x V_N),0\leq x\leq
1) \\
\law ((1-f(x,V_{N-1},M_1))\det(I_{N-1}-x V_{N-1}),0\leq x\leq
1),
\end{multline*}
with the suitable $f$ from \eqref{eq:DecompCharPolyOff}. Let the
logarithm be defined as in the Introduction (ie by continuity from
$x=0$). The previous equation then implies, as $f$ is continuous
in $x$,
\[
\log \det(I_N-x V_N)\law \log (1-f(x,V_{N-1},M_1))+\log
\det(I_{N-1}-x V_{N-1}).
\]
One can easily check that $|f(x,V_{N-1},M_1)|<1$ for all
$x\in[0,1]$ a.s., so
\[
\log (1-f(x,V_{N-1},M_1))= \sum_{j\geq 0}
(-1)^{j+1}\frac{f(x,V_{N-1},M_1)^j}{j}
\]
for all $x\in[0,1]$ almost surely. In particular, for $x=1$, we
get
\[
\log \det(I_N-V_N)\law  \left(\sum_{j\geq 0}
(-1)^{j+1}\frac{M_{11}^j}{j}\right)+\log\det(I_{N-1}-V_{N-1}),
\]
which gives the expected result by an immediate induction. We have
therefore shown that $\log Z_N\law X_N$.

{\it Proof for (2).} The idea is basically that $\beta_{1,k-1}$ tends in law to a Dirac distribution on 0 as $k$ tends to
$\infty$. So $\log (1+e^{\ii\theta_k}\sqrt{\beta_{1,k-1}})$ is well approximated by $e^{\ii\theta_k}\sqrt{\beta_{1,k-1}}$,
and as this has a distribution invariant by rotation, the central limit theorem will be easily proven from classical results
in dimension 1.

More precisely, $X_N$ can be decomposed as
\[
X_N = \underbrace{ \sum_{k=1}^N e^{\ii\theta_k} \sqrt{
\beta_{1,k-1}} }_{X_1(N)} - \frac{1}{2} \underbrace{ \sum_{k=1}^N
e^{2\ii\theta_k} \beta_{1,k-1} }_{X_2(N)} + \underbrace{ \sum_{j\geq
3} \sum_{k=1}^N \frac{(-1)^{j+1}}{j} \left(e^{\ii\theta_k}
\sqrt{\beta_{1,k-1}} \right)^{j} }_{X_3(N)}
\]
where all the terms are absolutely convergent. We study these three terms separately.

Clearly $X_1(N)$ has a distribution which is invariant by
rotation, so to prove that $\frac{X_1(N)}{\sqrt{\frac{1}{2}\log
N}} \overset{\mbox{law}}\longrightarrow \mathcal{N}_1+\ii
\mathcal{N}_2$, we only need to prove the following result for the
real part :
$$
\frac{\sum_{k=1}^N\cos \theta_k\sqrt{\beta_{1,k-1}}}{\sqrt{\frac{1}{2}\log N}}
\overset{\mbox{law}}\longrightarrow \mathcal{N},
$$
where $\mathcal{N}$ is a standard normal variable. As $\E(\cos^2(\theta_k)\beta_{1,k-1})=\frac{1}{2k}$, this is a direct
consequence of the central limit theorem (our random variables check the Lyapunov condition).

To deal with $X_2(N)$, as $\sum_{k\geq 0} 1/k^2<\infty$, there
exists a constant $c>0$ such as $\E(|X_2(N)|^2)<c$ for all
$N\in\NN$. Thus $(X_2(N),N\geq 1)$ is a $L^2$-bounded martingale,
so it converges almost surely. Hence
\[
X_2(N)/\sqrt{\tfrac12 \log N} \to 0 \quad a.s.
\]

Finally, for $X_3(N)$, let
$Y:=\sum_{j=3}^\infty\sum_{k=1}^\infty\frac{1}{j}(\beta_{1,k-1})^{j/2}$.
One can easily check that $\E(Y)<\infty$, so $Y<\infty$ a.s., so as
$N\to\infty$
\[
|X_3(N)|/\sqrt{\tfrac12 \log N} < Y/\sqrt{\tfrac12 \log N}\to 0
\quad a.s.
\]

Gathering all these convergences, we get the expected result :
$$
\frac{X_N}{\sqrt{\frac{1}{2}\log N}}
\overset{\mbox{law}}\longrightarrow \mathcal{N}_1+\ii \mathcal{N}_2,
$$
with the notations of Theorem 4.1.

\subsection{Proof from the decomposition in section \ref{section:DecompIndRV}.}
\label{subsection:CLTSecondProof}

We shall give here a very simple proof of the central limit
theorem for $\frac{\log Z_{N}}{\sqrt{\frac{1}{2} \log N}}$ based
on the decomposition into sums of independent random variables and
a classical version of the multidimensional central limit theorem.

From Theorem~\ref{thmfondamentaldedecomposition}, we have:
\begin{equation}\label{loijointe2}
\left(\Im\log Z_{N},\log|Z_{N}|\right) \law
\left(\sum_{j=1}^{N}W_{j},\sum_{j=1}^{N}T_{j}\right).
\end{equation}
where
\begin{equation}\label{e1}
T_{j}=\log(\beta_{j,j-1}2\cos W_{j}).
\end{equation}

Now, from the discussion preceding Theorem \ref{thmfondamentaldedecomposition}, we have for $s>-1$ and $t>-1$:
\begin{eqnarray}
\E\left[e^{\ii sW_{j}}\right] &=& \frac{\Gamma(j)^2}{\Gamma\left(j+\frac{s}{2}\right)\Gamma\left(j-\frac{s}{2}\right)}, \\
\E\left[e^{tT_{j}}\right] & =&  \frac{\Gamma\left(j\right)\Gamma\left(j+t\right)}{\Gamma\left(j+\frac{t}{2}\right)^{2}}.
\end{eqnarray}
From these Fourier transforms, one can easily deduce the moments
or the cumulants of all orders for $W_{j}$ and $T_{j}$ (see
\cite{Petrov95, ProkStat, Saulis} for definition of cumulants and
their relations with moments) by taking successive derivatives at
$0$. For our purpose, we will only need the first three moments or
cumulants. Since the calculation of the derivatives have already
been done in \cite{KeaSna}, we will only recap them here. Let us
call $Q_{j,k}$ the $k$-th cumulant of $T_{j}$ and $R_{j,k}$ the
$k$\textsuperscript{th} cumulant of $W_{j}$. Then we have
\begin{equation*}
Q_{j,k}=\frac{2^{k-1}-1}{2^{k-1}}\psi^{\left(k-1\right)}\left(j\right)
\end{equation*}
and
\begin{equation*}
R_{j,k}=\begin{cases}
0 & \text{if $k$ is odd} \\
\frac{\left(-1\right)^{k/2+1}}{2^{k-1}} \psi^{\left(k-1\right)}\left(j\right) & \text{if $k$ is even}
\end{cases}
\end{equation*}
where
\begin{equation*}
\psi^{\left(k\right)}\left(z\right)=\dfrac{\dd^{k+1}\log\Gamma\left(z\right)}{dz^{k+1}}
\end{equation*}
are the polygamma functions. Now, since the cumulants of a sum of
independent random variables are the sum of the cumulants, we can
easily obtain that the cumulants of $\sum_{j=1}^{N}T_{j}$ and
$\sum_{j=1}^{N}W_{j}$ are respectively
$$\frac{2^{k-1}-1}{2^{k-1}}\sum_{j=1}^{N}\psi^{\left(k-1\right)}\left(j\right)$$
and $$\begin{cases}
0 & \text{if $k$ is odd} \\
\frac{\left(-1\right)^{k/2+1}}{2^{k-1}}\sum_{j=1}^{N} \psi^{\left(k-1\right)}\left(j\right) & \text{if $k$ is even}
\end{cases}.$$
Moreover, we have the following expansion of the polygamma
function (see, for example Corollary 1.4.5 of \cite{AAR}):
\begin{equation}\label{e2}
\psi\left(z\right)\sim\log z-\dfrac{1}{2z}-\sum_{n=1}^{\infty}\dfrac{B_{2n}}{2nz^{2n}}
\end{equation}and
\begin{equation}\label{e3}
\psi^{\left(k\right)}\left(z\right)=\left(-1\right)^{k-1}\left[\dfrac{\left(k-1\right)!}{z^{k}}+\dfrac{k!}{2z^{k+1}}
+\sum_{n=0}^{\infty}B_{2n}\dfrac{\left(2n+k-1\right)!}{\left(2n\right)!z^{2n+k}}\right]
\end{equation}
for $|z|\to\infty$ and $|\arg z|<\pi$, and where the $B_{2n}$ are
the Bernoulli numbers. We deduce from \eqref{e2} that the
variances of $\Re \log Z_{N}$ and $\Im\log Z_{N}$ (which are centered)
are finite and both asymptotic to $\frac{1}{2}\log N$ as
$N\to\infty$.

Now we state the central limit theorem we shall apply (we follow
page 87 of Strook \cite{Stroock}). We assume that
$\left(X_{n}\right)$ is a sequence of independent, square
integrable $\mathbb{R}^{\ell}$ valued random variables defined on
the same probability space. Further we will assume that $X_{n}$
has mean $0$ and strictly positive covariance
$\cov\left(X_{n}\right)$. Finally we set:
$$S_{n}=\sum_{m=1}^{n}X_{m},\quad C_{n}:= \cov\left(S_{n}\right)=\sum_{m=1}^{n}\cov\left(X_{m}\right)$$and
$$\Sigma_{n}=\left(\det\left(C_{n}\right)\right)^{\frac{1}{2\ell}}\quad \text{and } \widehat{S}_{n}=\frac{S_{n}}{\Sigma_{n}}.$$

\begin{thm}[Multidimensional Central Limit Theorem, \cite{Stroock} p. 88]\label{MCLT}
Assume that $$A:= \lim_{n\to\infty}\frac{C_{n}}{\Sigma_{n}^{2}}$$exists and that
\begin{equation}\label{C1}
\lim_{n\to\infty}\frac{1}{\Sigma_{n}^{2}}\sum_{m=1}^{n}\E\left[|X_{m}|^{2}\I_{|X_{m}|\geq\varepsilon\Sigma_{n}}\right]=0
\end{equation}for every $\varepsilon>0$. Then the vector $\widehat{S}_{n}$ converges in law to a Gaussian vector with mean $0$ and covariance matrix $A$.
\end{thm}
Now we can prove Theorem 4.1 :

$$\frac{Z_N}{\sqrt{\frac{1}{2}\log N}}
\convlaw \mathcal{N}_1+ \ii \mathcal{N}_2.$$

Indeed, let us consider the Lyapounov exponents associated with
$(T_{n})$ and $(W_{n})$:
\[
L_{N}=\frac{1}{s_{N}^{3/2}}\sum_{n=1}^{N}\E\left[|T_{n}|^{3}\right],
\]
and
\[
L'_{N}=\frac{1}{\sigma_{N}^{3/2}}\sum_{n=1}^{N}\E\left[|W_{n}|^{3}\right],
\]
where $s_{N}^{2}=\sum_{j=1}^{N}\E\left[T_{j}^{2}\right]$ and
$\sigma_{N}^{2}=\sum_{j=1}^{N}\E\left[W_{j}^{2}\right]$. From the
expressions of the cumulants, we have:
\[
s_{N}^{2}=\sigma_{N}^{2}=\frac{1}{2}\sum_{j=1}^{N}\psi'\left(j\right)\sim
\frac{1}{2}\log N.
\]
It is not hard to see, using the expression for the cumulants or
the density of the beta variables and the $W_{j}$, that the series
$\sum_{n=1}^{\infty}\E\left[|T_{n}|^{3}\right]$ and
$\sum_{n=1}^{\infty}\E\left[|W_{n}|^{3}\right]$ both converge.
Hence $L_{N}\to 0$ and $L'_{N}\to 0$ as $N\to\infty$, and
consequently \eqref{C1} holds and the result follows from an
application of the Multidimensional Central Limit Theorem of
Theorem~\ref{MCLT}.

\section{Iterated logarithm law} \label{section:ILL}

In this Section, we give some iterated logarithm law for both the
real and imaginary parts of the characteristic polynomial. Again,
this can be done due to the decomposition given in
Theorem~\ref{thmfondamentaldedecomposition}.

We first need some information about the rate of convergence in
the central limit theorem.

\subsection{Further results about the rate of convergence}
\label{subsection:RateConv}

With the representation in
Theorem~\ref{thmfondamentaldedecomposition} it is possible to
obtain uniform and non-uniform estimates on the rate of
convergence in the central limit theorem, using the Berry-Essen's
inequalities (see \cite{Petrov95} or \cite{ProkStat}):
\begin{thm}\label{berryessenineq}
Let $X_{1},\ldots,X_{n}$ be independent random variables such that
$\E\left[X_{j}\right]=0$, and $\E\left[|X_{j}|^{3}\right]<\infty$.
Put $\sigma_{j}^{2}=\E\left[X_{j}^{2}\right]$;
$B_{n}=\sum_{j=1}^{N}\sigma_{j}^{2}$;
$F_{n}\left(x\right)=\mathbb{P}\left[B_{n}^{-1/2}\sum_{j=1}^{n}X_{j}\leq
x\right]$ and
\[
L_{n}=\frac{1}{B_{n}^{3/2}}\sum_{j=1}^{n}\E\left[|X_{j}|^{3}\right].
\]
Then there exist two constants $A$ and $C$ not depending on $n$
such that the following uniform and non uniform estimates hold:
\begin{equation}
\sup_{x}|F_{n}\left(x\right)-\Phi\left(x\right)|\leq A L_{n}
\end{equation}and
\begin{equation}
|F_{n}\left(x\right)-\Phi\left(x\right)|\leq\frac{CL_{n}}{\left(1+|x|\right)^{3}}.
\end{equation}
\end{thm}
Now, applying the above theorem to the variables
$\left(T_{j}\right)$ and $\left(W_{j}\right)$ we obtain:
\begin{prop}
The following estimates on the rate of
convergence in the central limit theorem  for the real and
imaginary parts of the characteristic polynomial hold:
\begin{eqnarray}
\left\vert\mathbb{P}\left[\frac{\Re \log
Z_{N}}{\sqrt{\frac{1}{2}\log N}}\leq
x\right]-\Phi\left(x\right)\right\vert
&\leq& \frac{C}{\left(\log N\right)^{3/2}\left(1+|x|\right)^{3}} \\
\left\vert\mathbb{P}\left[\frac{\Im \log
Z_{N}}{\sqrt{\frac{1}{2}\log N}}\leq
x\right]-\Phi\left(x\right)\right\vert &\leq& \frac{C}{\left(\log
N\right)^{3/2}\left(1+|x|\right)^{3}}
\end{eqnarray}
where $C$ is a constant.
\end{prop}

\subsection{An iterated logarithm law}
We first state a theorem of Petrov, \cite{Petrov66, Petrov99}.
\begin{thm}[Petrov] \label{lilpetrov}
Let $X_{1},X_{2},\ldots$ be independent random variables such that
$\E\left[X_{j}\right]=0$, and
$\sigma_{j}^{2}=\E\left[X_{j}^{2}\right]<\infty$. Set
$B_{n}=\sum_{j=1}^{N}\sigma_{j}^{2}$;
$F_{n}\left(x\right)=\mathbb{P}\left[B_{n}^{-1/2}\sum_{j=1}^{n}X_{j}\leq
x\right]$ and
$\Phi(x)=\frac{1}{\sqrt{2\pi}}\int_{-\infty}^{x}e^{-t^{2}/2}\dd
t$. If the conditions
\begin{enumerate}
  \item $B_{n}\to\infty$;
  \item  $\frac{B_{n+1}}{B_{n}}\to 1$;
  \item $\sup_{x}|F_{n}(x)-\Phi(x)|=\mathcal{O}\left(\left(\log B_{n}\right)^{-1-\delta}\right)$,
\end{enumerate} are satisfied for some $\delta>0$, then
\begin{equation}
\limsup\frac{S_{n}}{\sqrt{2B_{n}\log\log B_{n}}}=1\;\; \text{a.s.}
\end{equation}
\end{thm}
\begin{rem}
If the conditions of the theorem are satisfied, then we also have:
$$\liminf\frac{S_{n}}{\sqrt{2B_{n}\log\log B_{n}}}=-1\;\; \text{a.s.}$$
\end{rem}

Before using Theorem~\ref{lilpetrov} for the real and imaginary
parts of $\log Z_N$, we need to give the explicit meaning of the
``almost sure convergence'' for matrices with different sizes.

Imagine that in  Proposition~\ref{prop:HaarDecomp} we choose for
$M$ the symmetry which transforms $e_1$ (the first vector of the
basis) into $M_1$, a random vector of $\S_\CC^{n+1}$. Consider the
set $O=\S_\CC^1\times\S_\CC^2\times\S_\CC^3\dots$ endowed with the
measure $\nu_1\times\nu_2\times\nu_3\dots$, where $\nu_k$ is the
uniform measure on the sphere $\S_\CC^k$ (this can be a
probability measure by defining the measure of a set as the
limiting measure of the finite-dimensional cylinders). Consider
the application $f$ which transforms $\omega\in O$ into an element
of $U(1)\times U(2)\times U(3)\dots$ with successive iterations of
the Proposition~\ref{prop:HaarDecomp}. Then $\Omega=\Im(f)$ is naturally endowed with a
probability measure $\mu_U=\Im(\nu)$, and the marginal
distribution of $\mu_U$ on the $k$\textsuperscript{th} coordinate
is the Haar measures on $U(k)$.

Let $g$ be a function of a unitary matrix $U$, no matter the size
of $U$ (e.g. $g=\det(Id-U)$). The introduction of the set $\Omega$
with measure $\mu_U$ allows us to define the ``almost sure''
convergence of $(g(U_k),k\geq 0)$, where $(U_k)_{k\geq
0}\in\Omega$. This is, for instance, the sense of the ``a.s'' in
the following iterated logarithm law.

\begin{prop}
The following almost sure convergence (defined previously) holds :
\begin{eqnarray}
\limsup\frac{\Re \log Z_N}{\sqrt{\log N\log\log\log N}} &=& 1, \\
\limsup\frac{\Im \log Z_N}{\sqrt{\log N\log\log\log N}} &=& 1.
\end{eqnarray}
\end{prop}
\begin{rem}
The representations in law as sums of independent random variables we have obtained could as well be used to obtain all sorts of refined large and moderate deviations  estimates for the characteristic polynomial. 
\end{rem}

\section{Same results in the orthogonal case.}
\label{section:OrthogGroup}

The Mellin Fourier transform for $Z:=\det(I_N-M)$  found in
\cite{KeaSna} and \cite{KeaSna2} by Keating and Snaith, using the
Selberg integrals, are, for the unitary group $U(N)$
\begin{equation}\label{eq:7.1}
\E\left[|Z|^t e^{\ii s \arg
Z}\right]=\prod_{k=1}^N\frac{\Gamma\left(k\right)\Gamma\left(k+t\right)}{\Gamma\left(k+\frac{t+s}{2}\right)\Gamma\left(k+\frac{t-s}{2}\right)};
\end{equation}
and for the special orthogonal group $SO(2N)$
\begin{equation}\label{eq:7.2}
\E\left[Z^t\right]=2^{2Nt}\prod_{k=1}^N\frac{\Gamma\left(N+k-1\right)\Gamma\left(t+k-\frac{1}{2}\right)}{\Gamma\left(k-\frac{1}{2}\right)
\Gamma\left(t+k+N-1\right)}.
\end{equation}

Formula~\eqref{eq:7.1} was directly proven in
section~\ref{section:Haar}. Here we show that such a probabilistic
proof still holds for formula~\eqref{eq:7.2}.

Let
\[
\S_\RR^N:=\{(r_1,\dots,r_N)\in\RR^N : |r_1|^2+\dots+|r_N|^2=1\}
\]
and $\mu_{O(N)}$ be the Haar measure on $O(N)$. Then in analogy to
Proposition~\ref{prop:HaarDecomp}, we have:

\begin{prop}
Let $M\in O(N+1)$ be chosen such that its first column $M_1$ is
uniformly distributed on $\S_\RR^{N+1}$. Let $O_N\in O(N)$ be
chosen independently of $M$ according to the Haar measure
$\mu_{O(N)}$. Then the matrix
$$O_{N+1}:=M \mat{1}{0}{0}{O_N}
$$
is distributed with the Haar measure $\mu_{O(N+1)}$.
\end{prop}

The proof is essentially the same as the proof for
Proposition~\ref{prop:HaarDecomp}. If we choose for $M$ a symmetry
transforming $e_1$ in a uniformly chosen vector of
$\S_\RR^{(n+1)}$, we transform a random element of $SO(n)$ into an
element of $O(n+1)$ with determinant $-1$, and reciprocally. As a
consequence, the following result, analogous of
Proposition~\ref{prop:Decomp1}, can easily be shown.

\begin{cor}\label{cor:OrthogCharPoly}
Let $SO\in SO(2n)$ be distributed with the Haar measure
$\mu_{SO(2n)}$. Then
$$
\det (I_{2n}-SO)\law 2 \prod_{k=2}^{2n}
\left(1+\epsilon_k\sqrt{\beta_{\frac{1}{2},\frac{k-1}{2}}}\right),
$$
with
$\epsilon_1,\dots,\epsilon_{2n},\beta_{1/2,1/2},\dots,\beta_{1/2,(2n-1)/2}$
independent random variables such that
$\Prob(\epsilon_k=1)=\Prob(\epsilon_k=-1)=1/2$, and the $\beta$'s
being beta distributed with the indicated parameters.
\end{cor}

\begin{rem}
A direct calculation with the suitable change of variables shows
that
$$1+\epsilon_k\sqrt{\beta_{\frac{1}{2},\frac{k-1}{2}}}\law 2\beta_{\frac{k-1}{2},\frac{k-1}{2}},$$
from which formula~\eqref{eq:7.2} can be easily recovered.
\end{rem}

\begin{rem}
The same reasoning can be applied to many other groups such as
$\mathbb{H}(n)$, the set of $n\times n$ matrices $H_n$ on the
field of quaternions with $\overline{H_n}^TH_n=I_n$. The
symplectic group requires some additional work, to appear in a
future paper.
\end{rem}

\renewcommand{\refname}{References}

\end{document}